\documentclass[a4paper,reqno,11pt]{amsart}
\usepackage{amssymb}
\usepackage{amsmath}
\usepackage{fourier}
\usepackage{a4wide}
\usepackage{graphicx}
\usepackage{pstricks}
\usepackage{srcltx}
\usepackage{hyperref}

\theoremstyle{plain}
 \newtheorem{thm}{Theorem}[section]
 \newtheorem{prop}[thm]{Proposition}
 \newtheorem{lem}[thm]{Lemma}
 \newtheorem{cor}[thm]{Corollary}
\theoremstyle{definition}
 \newtheorem{exm}[thm]{Example}
 \newtheorem{dfn}[thm]{Definition}
\theoremstyle{remark}
 \newtheorem{rem}{Remark}[section]

 \numberwithin{equation}{section}

\renewcommand{\geq}{\geqslant}

\title[Periodic orbits in Hamiltonian systems with involutory symmetries]{Periodic orbits in Hamiltonian systems with involutory symmetries}



\author[R. Alomair]{\bfseries Reem Alomair}
\address{
School of Mathematics \\ 
The University of Manchester  \\ 
Manchester M13 9PL\\
United Kingdom}
\email{reem.alomair@postgrad.manchester.ac.uk}

\author[J. Montaldi]{\bfseries James Montaldi}
\address{
School of Mathematics \\ 
The University of Manchester  \\
Manchester M13 9PL\\
United Kingdom}
\email{j.montaldi@manchester.ac.uk}

\begin{document}

\vspace{18mm} \setcounter{page}{1} \thispagestyle{empty}

\begin{abstract}
We study the existence of families of periodic solutions in a neighbourhood of a symmetric equilibrium point in two classes of Hamiltonian systems with involutory symmetries. In both classes, involutions reverse the sign of the Hamiltonian function. In the first class we study a Hamiltonian system with a  reversing involution R acting symplectically. We first recover a result of Buzzi and Lamb showing that the equilibrium point is contained in a three dimensional conical subspace which consists of a two parameter family of periodic solutions with symmetry R and there may or may not exist two families of non-symmetric periodic solutions, depending on the coefficients of the Hamiltonian. 
In the second problem we study an equivariant Hamiltonian system with a symmetry S that acts anti-symplectically. Generically, there is no S-symmetric solution in a neighbourhood of the equilibrium point. Moreover, we prove the existence of at least 2 and at most 12 families of non-symmetric periodic solutions.  We conclude with a brief study of systems with both forms of symmetry, showing they have very similar structure to the system with symmetry R. 
\end{abstract}

\maketitle

\section{Introduction}

A classical approach in the analysis of Hamiltonian systems is to study the existence of periodic orbits near  equilibria. A basic theorem on the existence of periodic solutions in Hamiltonian systems is the Liapunov centre theorem, which states that if the linearized flow at an equilibrium point has a simple purely imaginary eigenvalue satisfying a non-resonance condition then there exists a smooth 2-dimensional manifold which passes through the equilibrium point and consists of a one parameter family of periodic solutions, or nonlinear normal mode.  In this work we  extend this theorem to two classes of Hamiltonian systems with involutory symmetries where in both cases the involution reverses the sign of the Hamiltonian. In the first case, already studied by Buzzi and Lamb \cite{b1}, the involution is symplectic while in the second case it is anti-symplectic.

 In the literature, there are versions of  the Liapunov centre theorem for reversible systems, but they mostly deal with the classical case where the reversing symmetry acts anti-symplectically. For example see Devaney \cite{b13}.  In this paper we consider two types of symmetry. First is the existence of periodic solutions in a time reversing Hamiltonian system equipped with an involution $R$ that acts symplectically. The problem was introduced and analysed by Buzzi and Lamb \cite{b1}. If the linear system has two pairs of purely imaginary eigenvalues, they prove in a neighbourhood of a symmetric equilibrium point the existence of a three dimensional subspace consists of a two parameter family of $R$-periodic solutions with period close to $2\pi$. In addition, they claim to find two families of non-symmetric periodic solutions whose period tends to $2\pi$ as they  approach the equilibrium point, for an open dense set of coefficients (however there is a sign error in one of their calculations).  Motivated by this work, we looked at the problem using different coordinates, and hence a different set of invariants.  We recover their result on the existence of symmetric periodic solutions but obtain a different conclusion for the non-symmetric solutions. We determine an expression in the fourth order normal form and show that if this expression is positive there are two families of non-symmetric solutions, while if it is negative there are none.  

The second problem we discuss is the dynamics near an equilibrium point in an equivariant Hamiltonian system with an involutory  (time preserving) symmetry $S$ acting anti-symplectically.   Bifurcations of equilibria in Hamiltonian systems with such symmetry have been considered recently by M.~Bosschaert and H.~Han{\ss}mann \cite{BH}.   Existence theorems for  periodic solutions in symmetric Hamiltonian systems  can be found in Montaldi et al \cite{b8}, \cite{b9}, but this and related work assumes the symmetry transformation acts symplectically.  We prove that for systems with this anti-symplectic symmetry, generically,  there are no symmetric periodic orbits in a neighbourhood of an equilibrium point.  Moreover, we prove the existence of at least 2 and at most 12 non-symmetric families of periodic solutions (nonlinear normal modes) in a neighbourhood of the equilibrium point under the same generic conditions.  
 
In both cases, since the involution reverses the sign of the Hamiltonian and we assume the linear system is periodic, the equilibrium will be in 1:-1 resonance.

The paper  is organised as follows. In Section \ref{s1} we introduce basic facts and definitions of Hamiltonian systems with symmetry. Section 3 lists normal forms of the Hamiltonian linear system $L$, the structure map $J$ and the symmetry elements $R$ and $S$ in $\mathbb{C}^2$. Section 4 reviews the standard tool used to find periodic orbits in Hamiltonian systems:  Liapunov-Schmidt reduction. In Section 5 we state and prove our theorem on the existence of families of periodic orbits in the $R$-reversible Hamiltonian system with $R$ acting symplectically. In Section 6 we give our main result on the existence of periodic solutions in the $S$-equivariant Hamiltonian system with $S$ acting anti-symplectically. Finally, in Section 7 we study the existence of periodic solutions in systems with the combined symmetry $\mathbb{Z}_2^R \times\mathbb{Z}_2^S$ reversible/equivariant Hamiltonian system.

\section{Hamiltonian systems with symmetry}
\label{s1}
In this section we recall  some basic facts and definitions on Hamiltonian systems with symmetry.
 Let $(\mathbb{R}^{2n},\omega)$ be a symplectic space, i.e  an even dimensional vector space  equipped with a symplectic form $\omega$. Recall that a symplectic form is a non-degenerate, skew symmetric, bilinear form.  Then there exists a structure map $J$ satisfying $J^*=-J$ ($J^*$ denotes the transpose of $J$) and $J^2=-I$ such that $\omega(x,y)=\langle x,Jy \rangle$ for $x,y \in \mathbb{R}^{2n}$, where $\langle .,. \rangle$ is the standard inner product in $\mathbb{R}^{2n}$. Let $H:\mathbb{R}^{2n}\rightarrow\mathbb{R}$ be a Hamiltonian function. The Hamiltonian vector field $f$ generated by $H$ is symplectic, i.e.\ its flow preserves the symplectic form $\omega$, and is defined by

\begin{equation}
\label{p1}
\dot{x}=f(x)=J\nabla H.
\end{equation}
By using canonical coordinates for the symplectic form $\omega$ given in Darboux theorem \cite{b6} one can write
\[J=\left( \begin{array}{cc}
0&-I_n\\
I_n&0
\end{array}\right).\]

 In this work we will deal with two types of symmetry, equivariant symmetries and time-reversing symmetries.
 \begin{dfn}
 Let $S,R$ be two linear transformations of $\mathbb{R}^{2n}$, then
 \begin{enumerate}
 \item  The  vector field $f$ is called $S$-equivariant if
 \[f(Sx)=Sf(x)\, , \forall x\in \mathbb{R}.\] If $x(t)$ is a solution of \eqref{p1}, then $Sx(t)$ is also a solution and $S$ is referred to as a symmetry.
 \item The vector field $f$ is called $R$-reversible if
 \[f(Rx)=-Rf(x)\, , \forall x\in \mathbb{R}.\] If $x(t)$ is a solution of \eqref{p1}, then $Rx(-t)$ is also a solution. Such a transformation is called a time reversing symmetry.
\end{enumerate}
\end{dfn}
The symmetry of a periodic solution is given by the following definition.

\begin{dfn}
Let $x(t)$ be a periodic solution of the  dynamical system $\dot{x}=f(x)$.
\begin{enumerate}
\item If $S$ is a symmetry of the system then $x(t)$ is said to be $S$-symmetric if\[Sx(t+\theta)=x(t),\]
for some $\theta \in S^1$.
\item If $R$ is a reversing symmetry of the system then $x(t)$ is said to be $R$-symmetric if \[Rx(\theta-t)=x(t),\]
for some $\theta \in S^1$.
\end{enumerate}
Here we identify $S^1$ with $\mathbb{R}/T\mathbb{Z}$, where $T$ is the period of $x(t)$.  In both cases, a symmetric periodic orbit symmetric if and only if it is set-wise invariant.
\end{dfn}

In the Hamiltonian context, (reversing) symmetries can arise in two ways: they can either be symplectic or antisymplectic.
A (reversing) symmetry $T$ is symplectic if $\omega(Tx,Ty)=\omega(x,y),\forall x,y\in \mathbb{R}$ and anti-symplectic if $\omega(Tx,Ty)=-\omega(x,y),\forall x,y\in \mathbb{R}$. In matrix form we can choose a basis so that $T$ is orthogonal, and then $T$ is symplectic if $TJ= JT$ and anti-symplectic if $TJ=- JT$.

Note for example that by \eqref{p1}, if a reversing symmetry $R$ is symplectic then it must reverse $\nabla H$, and if we assume (as we may, and do) that $H(0)=0$ then this is equivalent to $H(Rx) = -H(x)$, so that $H$ is `anti-invariant'.   There are in all 4 possibilities of symmetry, labelled as follows

\begin{table}[h]
\begin{tabular}{c|ccc}
  type & $\omega$ & $f$ & $H$ \\
\hline
SE & +1 & +1 & +1 \\
AR & -1 & -1 & +1 \\
SR & +1 & -1 & -1 \\
AE & -1 & +1 & -1 \\
\end{tabular}
\caption{The `type' refers to a transformation being Symplectic-Equivariant, or Antisymplectic-Reversing etc.}
\end{table}

Note that if $T$ is an involution which reverses the sign of $H$, then any symmetric periodic orbit must lie in the set where $H=0$.  There may on the other hand be periodic orbits on which $H$ is non-zero, and then $T$ will exchange two such orbits, one with $H>0$ and the other with $H<0$.  We will see this in more detail in later sections.

\section{Linear Hamiltonian systems with involutory symmetries}

In this section we give the normal forms of linear Hamiltonian systems with involutory symmetries. Recall that an involution is a transformation of order $2$. An important assumption that is required for studying the existence of periodic orbits is the presence of purely imaginary eigenvalues of the linear Hamiltonian vector field.

 Let $L\in sp_J(2n,\mathbb{R})$ be a linear Hamiltonian vector field. Thus,
\[LJ=-JL^*,\]
where $J$ is the structure map defined in the previous section. By Bochner's theorem \cite{b5}, a (reversing) symmetry $T$ can be chosen to be linear and orthogonal. Therefore, the (reversing) equivariant condition can be written as
\[LT=\pm TL,\]
and the (anti-)symplectic property of $T$ is given by
\[TJ=\pm JT.\]

In \cite{b3}, Hoveijn et al. gave normal forms of linear systems in eigenspaces of (anti-) automorphisms of order two, which can be adapted to our problem. These normal forms are based on writing $\langle J,T \rangle-$ invariant subspaces. Since we are interested in generic systems with given symmetry, then by \cite{b3} we can only focus on the case when $L$ is semi-simple. Also, we assume that $L$ has at least one pair of purely imaginary eigenvalues $\pm i$. Normal forms of $T,J$ and $L$ are given in the following lemma.
We use the notation,
$$
  I_2=\begin{pmatrix}1&0\cr 0&1\end{pmatrix},\quad
  J_2=\begin{pmatrix}0&-1\cr 1&0\end{pmatrix},\quad\hbox{and}\quad
  S_2=\begin{pmatrix}1&0\cr 0&-1\end{pmatrix}.
$$

\begin{lem}
\label{l1}
Let $L$ be a linear Hamiltonian vector field on $\mathbb{R}^{2n}$.
\begin{enumerate}
\item[i)] Suppose $L$ is $R$-reversible, with $R$ acting symplectically (symmetry type SR).\\
Let $V$ be a minimal $(L,J,R)$-invariant subspace on which $L$ has eigenvalues $\pm i$. Then $\dim V=4$ and $R|_V,J|_V$ and $L|_V$ can take the following normal forms
\[
R|_V=\left(
\begin{array}{cc}
0&I_2\\
I_2&0\\
\end{array}
\right), \quad
J|_V=\left(
\begin{array}{cc}
J_2&0\\
0&J_2\\
\end{array}
\right), \quad\hbox{and}\quad
L|_V=\left(
\begin{array}{cc}
J_2&0\\
0&-J_2\\
\end{array}
\right).
\]

\item [ii)]Suppose now $L$ is $S$ equivariant, with $S$ acting anti-symplectically (symmetry type AE).\\
Let $V$ be a minimal $(L,J,S)$-invariant subspace on which $L$ has eigenvalues $\pm i$. Then $\dim V=4$ and $S|_V,J|_V$ and $L|_V$ can take the following normal forms
\[
S|_V=\left(
\begin{array}{cc}
0&S_2\\
S_2&0\\
\end{array}
\right),
\quad J|_V=\left(
\begin{array}{cc}
J_2&0\\
0&J_2\\
\end{array}
\right), \quad\mbox{and}\quad
L|_V=\left(
\begin{array}{cc}
J_2&0\\
0&-J_2\\
\end{array}
\right).
\]
\end{enumerate}
\end{lem}
\begin{proof}
Normal forms (i) are given in \cite{b1}. For (ii),  Let $W$ be a 2-dimensional symplectic subspace on which $L$ has the pair of eigenvalues $\pm i$ and $S(W)=W$. It is known in Hamiltonian context that $L$ and $J$ can take the same normal form on $W $ taking into account multiplication of time by a scalar. Equivariance property yields $SL=LS$. On $W$, $L$ and $J$ take the same form which gives $SJ=JS$ which contradicts the fact that $S$ is acting anti-symplectically. Thus, the minimal invariant subspace is four dimensional and is given by $V=W\oplus W', \, W'=S(W)$. The anti-symplectic property implies $J|_{W'}=-J|_{W}$ while equivariance gives $L|_{W'}=L|_{W}=J|_{W}$. Therefore, normal forms given in \cite{b3} show
\[
S|_V=\left(
\begin{array}{cc}
0&I_2\\
I_2&0\\
\end{array}
\right), \quad
J|_V=\left(
\begin{array}{cc}
J_2&0\\
0&-J_2\\
\end{array}
\right), \quad\hbox{and}\quad
L|_V=\left(
\begin{array}{cc}
J_2&0\\
0&J_2\\
\end{array}
\right).
\]
 To get the same formulas for $J$ and $L$ given in (i) apply the change of coordinates on $\mathbb{C}^2$ given by
 \[z_1=w_1, z_2=\bar{w_2}.\]
In these new coordinates $S,J$ and $L$ takes the forms given in (ii).
\end{proof}

Note that with these conventions, $L$ and $J$ take the same form in both cases, and the quadratic part $H_2$ of the Hamiltonian in both is given by
$$H_2(z_1,\,z_2) = |z_1|^2 - |z_2|^2;$$
that is, $H$ has a 1:-1 resonance.  The higher order terms will differ for the two cases, as we see below.

\section{Liapunov-Schmidt reduction}
The classical approach to finding periodic orbits in Hamiltonian systems is to solve a variational equation on the loop space. This equation is of  infinite dimension and can be reduced by Liapunov-Schmidt Reduction. In this section we will give an overview of that method and how to use it in finding periodic orbits near an equilibrium point in a reversible equivariant Hamiltonian system. We chose the reversible equivariant case to cover all symmetry cases discussed in this paper. We will follow the settings given in \cite{b1} and \cite{b4}.

Consider the vector field $f:\mathbb{R}^n\rightarrow\mathbb{R}^n,$ which has an equivariant reversing symmetry group $G$. This implies the existence of a representation $\rho:G\rightarrow O(n)$ and a reversing sign $\sigma:G\rightarrow \{\pm 1\}$ such that
\[f\rho(g)=\sigma(g)\rho(g)f,\forall g \in G.\]

In the following we give briefly the main steps of the Liapunov-Schmidt reduction and details can be found in \cite{b1}.

\subsection{Defining the operator $\Phi$}
Let $\Phi:\mathcal{C}_{2\pi}^1\times \mathbb{R}\rightarrow\mathcal{C}_{2\pi}$  be given by
\begin{equation}
 \Phi(u,\tau)=(1+\tau)\displaystyle\frac{du}{dt}-f(u)
  \end{equation}
 where $\mathcal{C}_{2\pi}$ is the Banach space of $\mathbb{R}^n$-valued continuous $2\pi$-periodic functions and $\mathcal{C}_{2\pi}^1$ is the space of $\mathcal{C}_{2\pi}$ functions that are continuously differentiable. It is readily seen that zeros of $\Phi$ are periodic solutions of the  dynamical system generated by $f$ with period $\frac{2\pi}{1+\tau}$. Now we can define the group action on the loop space $\mathcal{C}_{2\pi}$ as follows
 \[T:\widetilde{G}\times \mathcal{C}_{2\pi}\rightarrow \mathcal{C}_{2\pi}\]
 \[(T_gu)(t)=\rho(\gamma)(u(\sigma(\gamma)t+\theta)),\]
where $g=(\gamma,\theta)$ is an element of $\tilde{G}=G\ltimes S^1$. Straightforward calculations imply that the operator $\Phi$ is $\tilde{G}$- reversible equivariant, that is
\[\Phi(T_gu,\tau)=\sigma(\gamma)T_g\Phi(u,\tau),\quad \forall g=(\gamma,\theta)\in\widetilde{G}.\]

The linear part of $\Phi$ is defined by
 \[\mathcal{L}=(d\Phi)_{(0,0)}.\]

 It is readily verified that $\mathcal{L}$ is $\widetilde{G}$-reversible equivariant.

\subsection{The splittings}
  Consider the splittings
  \begin{equation}
  \label{p2}
  \mathcal{C}_{2\pi}^1=\ker{\mathcal{L}}\oplus{(\ker{\mathcal{L}})}^\bot\,\text{and}\,\mathcal{C}_{2\pi}={(\mathrm{range} { \mathcal{L}})}^\bot\oplus \mathrm{range}{\mathcal{L}},
  \end{equation}
  where the complements are taken with respect to the inner product
   \[ [u,v]=\int_{\tilde{G}} \langle T_gu,T_gv \rangle d\mu,\]
  where $\mu$ is a normalized Haar measure for $\widetilde{G}$ and $\langle u,v \rangle=\int_0^{2\pi}[u(t)]^t v(t)\,dt$. The splittings \eqref{p2} are $T_g$-invariant.
   Now we define the projections
  \[E:\mathcal{C}_{2\pi}\rightarrow \mathrm{range}{\mathcal{L}}\]
  \[I-E:\mathcal{C}_{2\pi}\rightarrow {(\mathrm{range}{ \mathcal{L}})}^\bot.\]
  Invariance  of \eqref{p2} under $T_g$ implies that the projections $E$ and $I-E$ commute with $T_g$.
  We start this step by solving the equation
   \[E\Phi(v+w,\tau)=0,\]
  for $w$ by the implicit function theorem, where $u=v+w,v\in\ker{\mathcal{L}}\, , w\in{(\ker{\mathcal{L}})}^\bot$. The solution $W=W(v,\tau)$ commutes with $T_g$. Thus, the Liapunov-Schmidt method reduces the original problem to the problem of finding the zeros of the bifurcation map which is defined by
      \[\varphi:\ker{\mathcal{L}}\times \mathbb{R}\rightarrow {(\mathrm{range}{\mathcal{L}})}^\bot\,\]
      \[\varphi(u,\tau)=(I-E)\Phi(v+W(v,\tau),\tau).\]
       An important property of the bifurcation map $\varphi$ is $\widetilde{G}$ reversing-equivariance property, i.e
      \[\varphi(T_gu,\tau)=\sigma(\gamma)T_g\varphi(u,\tau),\forall g\in \widetilde{G}.\]
 The last feature to be considered is the Hamiltonian structure of the bifurcation map.  Using the implicit Hamiltonian constrain given in \cite{b4} and \cite{b1} one can show that $\Phi$ is a parameter dependent Hamiltonian vector field.

 According to the actions of $G$ being (anti-)symplectic we define the symplectic sign $\chi$ by the homomorphism $\chi:G\rightarrow\{\pm 1\}$ such that
\[\omega(\gamma x,\gamma y)=\chi(\gamma)\omega(x,y),\gamma\in G.\]
Therefore, the weak symplectic form $\Omega$ will satisfy
\[\Omega(g u,g v)=\chi(\gamma)\Omega( u,v),g=(\gamma,\theta)\in\widetilde{G},\]
and the Hamiltonian sign is given by
\begin{equation}
\label{hamil}
\mathcal{H}(g u,g v)=\sigma(\gamma)\chi(\gamma)\mathcal{H}(u,v).
\end{equation}

In all cases we discuss $\ker{\mathcal{L}}$ is finite dimensional and thus $\ker {\mathcal{L}}=\ker {\mathcal{L}^*}$ and so by \cite[Theorem 6.2]{b4} the bifurcation equation is a Hamiltonian vector field. Its corresponding Hamiltonian $h$ satisfies the (semi-)invariance properties given in \eqref{hamil} restricted to $\ker{\mathcal{L}}$ i.e.
\begin{equation}
\label{hamil2}
h(gu)=\sigma(\gamma)\chi(\gamma)h(u), u\in \ker{\mathcal{L}},
\end{equation}
where as before $g=(\gamma,\theta)$ for some $\theta\in S^1$.  In practice, the function $h$ can be computed to any finite degree by using normal form transformations, as described for example in \cite{b9} (the discussion there is for symplectic symmetries, but is equally valid for all four cases listed in Table 1).

\section{Symplectic time-reversing involution}
In this section we prove the existence of symmetric and non-symmetric periodic solutions in a Hamiltonian systems with a reversing involutory symmetry acting symplectically (type SR in Table 1). The problem was first studied by C. Buzzi and J. Lamb \cite{b1}, but there is a minor sign error in the calculations in Lemma 6.4 which effects the statement in their Theorem 6.1. They (correctly) prove the existence of a three dimensional conical subspace of symmetric periodic solutions in a neighbourhood of the origin. Also, they find that the origin is contained in two 2-dimensional manifolds each containing a  non-symmetric family of periodic solutions with period close to $2\pi$. Using our expressions for the (semi-)invariants, we first recover their result on the symmetric solutions,  and then we correct their Theorem 6.1 to show that generically there may or may not be two families of non-symmetric periodic orbits in a neighbourhood of the equilibrium point 0 depending on the coefficients of the Hamiltonian.  Buzzi and Lamb also distinguish between two cases, called elliptic and hyperbolic, distinguishing between the possibilities of the period function on the 3-dimensional family being monotonic or not.  It turns out that this distinction coincides with the two cases of existence or non-existence of non-symmetric periodic orbits.

By the normal forms given in Lemma \ref{l1} (i), we have $\dim \ker \mathcal{L}=4$, so we can write $\ker \mathcal{L}\cong \mathbb{C}^2$. Therefore, the bifurcation map is given by
\begin{align*}
\varphi:\mathbb{C}^2 \times \mathbb{R}\rightarrow \mathbb{C}^2\\
\varphi=2J \nabla_{z}h
\end{align*}
 with Hamiltonian function
\begin{equation*}
h:\mathbb{C}^2 \times \mathbb{R}\rightarrow \mathbb{R},
\end{equation*}
which satisfies \eqref{hamil2}.
Denote by $\mathbb{Z}_2^R$ the cyclic group generated by $R$, which together with $S^1$ gives $S^1\rtimes \mathbb{Z}_2^R$. The reversing symmetry $R$ acts on  $\mathbb{C}^2$ by
\begin{equation*}
R(z_1 ,z_2)=(z_2 ,z_1)
\end{equation*}
while the $S^1$ action is defined by
\begin{equation*}
\theta(z_1 ,z_2)=(e^{i\theta} z_1 ,e^{-i\theta}z_2).
\end{equation*}
Let $\mathcal{E}$ be the ring of $S^1$ invariants, then one can write
\[\mathcal{E}=\mathcal{E}_+\oplus\mathcal{E}_-,\]
where $\mathcal{E}_+$ consists of $\mathbb{Z}_2^R$ invariants and $\mathcal{E}_-$ consists of $\mathbb{Z}_2^R$ anti-invariants.
\begin{lem}
\label{l2}
Let $S^1\rtimes\mathbb{Z}_2^R$ act on $\mathbb{C}^2$ as above, then
\begin{enumerate}
\item $\mathcal{E}$ is the ring generated by $A,B,C,D$ where
$A=|z_1|^2, B=|z_2|^2, C+iD= 2z_1 z_2$.
\item $\mathcal{E}_+$ is the subring of $\mathcal{E}$ generated by $N,C,D$ where
$ N=|z_1|^2+ |z_2|^2 $, and $\mathcal{E}_-$ is the module over $\mathcal{E}_+$  generated by the function $\delta =|z_1|^2-|z_2|^2$.
\item The orbit map $O: \mathbb{C}^2\rightarrow\mathbb{R}^3$ defined by $(z_1,z_2)\rightarrow(N,C,D)$ has image
$$\left\{(N,C,D) \mid N^2 \geq C^2 +D^2\right\}.$$
\end{enumerate}

Note that the functions $N,C,D$ and $\delta$ satisfy the identity $\delta^2=N^2 -C^2 -D^2$.
\end{lem}

The proof of this lemma is by standard algebraic computations, similar to those found for example in \cite{b11}.

Now we can apply Lemma \ref{l2} to our Hamiltonian. The function $h$ is $S^1$-invariant, $R$ anti-invariant and real valued. This implies there is a smooth function $g$ satisfying
\begin{equation}
\label{no45}
h(z_1,z_2,\tau) =\delta \ g(N,C,D,\tau).
\end{equation}

In order to find the periodic solutions we need to solve the bifurcation equation first. The bifurcation equation is given by
\begin{equation*}
\nabla _z h=0
\end{equation*}
This can be written as
\begin{equation} \label{SR bif eqn}
\left\{\begin{array}{rcccl}
\displaystyle\frac{\partial h}{\partial z_1} &=& \displaystyle\bar{z_1} g + \delta \frac{\partial g}{\partial z_1} &=& 0,\\[10pt]
\displaystyle\frac{\partial h}{\partial z_2} &=& \displaystyle - \bar{z_2}g + \delta \frac{\partial g}{\partial z_2} &=& 0.
\end{array}\right.
\end{equation}

We now consider, in turn, the symmetric and non-symmetric periodic orbits.

\subsection{Symmetric Periodic Orbits}

In finding symmetric periodic orbits we recover the result in \cite{b1}.
\begin{thm}[Buzzi \& Lamb \cite{b1}]
 Consider a symmetric equilibrium $0$ of a reversible Hamiltonian vector field $f$ with the reversing involution acting symplectically. Suppose that $Df(0)$ has two purely imaginary pairs of eigenvalues $\pm i$ with no other eigenvalues of the form $\pm ki,k \in \mathbb{Z}$. Then, the equilibrium is contained in a three-dimensional flow invariant conical subspace, given by the equation $\delta=0$, and generically this consists of a two-parameter family of symmetric periodic solutions whose period tends to $2\pi$ as they approach the equilibrium.
\end{thm}

\begin{proof}
Since the Hamiltonian is $R$ anti-invariant then all symmetric solutions are zeros of the bifurcation equations that lie in the level set $h=0$.
For symmetric solutions we have $\delta =0$. Therefore the bifurcation equation calculated in $\mathrm{Fix} R=\{(z,z)\mid z\in \mathbb{C}\}$ will take the form
\begin{equation*}
\bar z g(z,\tau) =0
\end{equation*}
Non-zero solutions yield $g(z,\tau)=0$. By the formula of the reduced  Hamiltonian \eqref{no45}, the lowest order term of the variable $\tau$ is given by
\begin{equation*}
h=(|z_1|^2-|z_2|^2)\frac{\tau}{2}+h.o.t.
\end{equation*}
This implies  that $\displaystyle{\frac{\partial g}{\partial \tau}}(0,0)=\frac{1}{2}\neq 0$. By the implicit function theorem for each small non-zero $z$ there exists a $\tau$ such that $(z,z)$ lies in a periodic orbit with period $\frac{2\pi}{\tau +1}$. By reversing property each $R$ symmetric solution intersects $\mathrm{Fix}R$ in two points. Since the conical subspace $\delta=0$ is 3 dimensional and all points in $\mathrm{Fix} R$ are solutions of the bifurcation equation we conclude that the conical subspace completely consists of these periodic solutions with period close to $2\pi$ as they approach the origin.
\end{proof}

\subsection{Non-Symmetric Periodic Orbits }
We prove the existence of two families of non-symmetric periodic solutions under suitable conditions on the coefficients of the Hamiltonian. This result is fairly different to the one in \cite{b1}.
To prove the existence of non-symmetric solutions one needs to solve the bifurcation equation without any symmetry conditions. By calculating the  partial derivatives of $g$ the bifurcation equation will be
\begin{align*}
\frac{\partial h}{\partial z_1} &= \bar{z_1}(g+ \delta g_N)+ z_2 \delta (g_C -ig_D)=0\\
\frac{\partial h}{\partial z_2} &= \bar{z_2}(-g+ \delta g_N)+ z_1\delta (g_C -ig_D)=0
\end{align*}
where $g_N=\displaystyle\frac{\partial g}{\partial N}$, $g_C=\displaystyle\frac{\partial g}{\partial C}$ and $g_D=\displaystyle\frac{\partial g}{\partial D}$.
 Multiplying the first equation by $z_1$ and  the second one by $z_2$ we get
\begin{align}
\label{no1}
 |z_1|^2(g+ \delta g_N)+ z_1z_2 \delta (g_C -ig_D)&=0\\
\label{no2}
 |z_2|^2(-g+ \delta g_N)+ z_1z_2\delta (g_C -ig_D)&=0
\end{align}
By adding \eqref{no1} and \eqref{no2} we have
\begin{equation}
\label{no3}
\delta (g + N g_ N+(C+iD)(g_C -ig_D))=0
\end{equation}
Taking  the imaginary part  of the above equation gives
\begin{equation*}
D g_C - C g_D=0
\end{equation*}
and when $C,D\neq 0$ we can write that equation as
\begin{equation}
\label{no4}
\frac{g_C}{C}=\frac{g_D}{D}
\end{equation}
Therefore equation \eqref{no3} will be
\begin{equation}
\label{no6}
\delta (g + N g_ N+ C g_C + D g_D )=0
\end{equation}
 By subtracting \eqref{no2} from \eqref{no1} we have
\begin{equation}
N g+\delta^2 g_N=0
\end{equation}
this can also be written by the formula
\begin{equation}
\label{no5}
\frac{g}{\delta^2}=-\frac{g_N}{N}
\end{equation}
Substituting \eqref{no4} and \eqref{no5} in \eqref{no6} yields
\begin{equation}
\label{no7}
\frac{g_N}{N}=-\frac{g_C}{C}
\end{equation}
Thus
\begin{equation}
\frac{g}{\delta^2} = -\frac{g_N}{N}=\frac{g_C}{C}=\frac{g_D}{D}
\end{equation}
which is equivalent to
\begin{equation}
\label{no14}
 \frac {N}{g_N}=-\frac {C}{g_C}=-\frac{D}{g_D},
\end{equation}

In order to prove the existence of non-symmetric periodic solutions to the original Hamiltonian system we need to prove the following lemma. Let \[g_N(0)=n, g_C(0)=c, g_D(0)=d.\]
\begin{lem}
\label{l3}
If $n,c$ and $d$ are not all zero then there exists a unique solution in $\mathbb{R}^4\cong(\tau,N,C,D)$-space  for the system of equations
\begin{align}
\label{no10}
g + N g_ N+ C g_C + D g_D&=0\\
\label{no11}
N g_C+C g_N&=0\\
\label{no12}
D g_C-C g_D&=0\\
\label{no13}
N g_D+D g_N&=0.
\end{align}
\end{lem}
\begin{proof}
 It is clear that the last three equations are not independent but we will use them all to make up for the special cases when one of the numbers $n,c$ or $d$ is equal to zero. Suppose that $n\neq 0$. Then we only need to solve \eqref{no10},\eqref{no11} and \eqref{no13}. In order to apply the implicit function theorem we need to study the following Jacobian matrix with respect to $\tau,C,D$ and $N$
 \[
\left(
\begin{array}{ccc|c}
 \frac{1}{2}&c&d&n\\
 0&n&0&c\\
 0&0&n&d\\
  \end{array}
   \right) =
   \left(
 \begin{array}{c|c}
 X&Y
  \end{array}
 \right)
 \]
 Since $n\neq 0$ then the matrix $X$ is non-singular. Therefore by the implicit function theorem there exists a unique curve $S=S(N)$, with $dS(0)=-X^{-1} Y$, that solves the system. If $n=0$ but $c\neq0$ we can choose equations \eqref{no10},\eqref{no11} and \eqref{no12}. Solving by the implicit function theorem gives a unique solution $S=S(C)$. A similar argument can be used for the remaining cases.

\end{proof}
Now we state and prove the main theorem about the existence of non-symmetric periodic solutions for the given reversible Hamiltonian system.

\begin{thm}
\label{th1}
Suppose that $n^2 \neq c^2+ d^2$, then there exist  the symmetric Liapunov centre families of periodic  solutions filling the set $\delta=0$ described before. Moreover,
\begin{enumerate}
\item [i)] If $n^2>c^2+d^2 $ then there exists two families of non-symmetric periodic orbits for the Hamiltonian system distinguished by the sign of $\delta$. The period of the periodic solutions converges to $2\pi$ as the solutions tend to the origin.
\item [ii)] If $n^2<c^2+d^2 $ then the only periodic orbits with period close to $2\pi$ in a neighbourhood of the origin are the symmetric ones.

\end{enumerate}
\end{thm}
\begin{proof}
To prove the existence of non-symmetric periodic orbits we have to solve the equations \eqref{no10},\eqref{no11},\eqref{no12} and \eqref{no13}. By the condition $n^2 \neq c^2+ d^2$ we have that $n,c$ and $d$ cannot all be zero. Applying Lemma \ref{l3} we have a unique solution for those equations. Therefore, we can write
\begin{equation}
 \frac {N}{g_N}=-\frac {C}{g_C}=-\frac{D}{g_D}=t,
\end{equation}
which is equivalent to $N= g_N t, C=-g_C t$ and $D=-g_D t$. To get non-symmetric solutions we should have $\delta^2=N^2-C^2-D^2 >0$. This implies
\begin{equation*}
({g_N}^2-{g_C}^2-{g_D}^2 )t^2>0, \ \text{for}\ t\neq0,
\end{equation*}
and therefore, ${g_N}^2-{g_C}^2-{g_D}^2>0$. Taking the limit at the origin gives $n^2\geq c^2+d^2 $.
We conclude that non-symmetric solutions exist when $n^2>c^2+d^2 $ and split into two families according to $\delta$ being positive or negative. On the other hand, when $n^2< c^2+d^2 $ the only periodic orbits with period close to $2\pi$ in a neighbourhood of the origin are the symmetric ones.
\end{proof}

\subsection{Period Distribution within the Family of Symmetric Periodic Solutions}
Following  the argument given in Buzzi and Lamb \cite{b1}, we describe the structure of period distribution for symmetric periodic solutions. According to $\mathrm{Fix} R$ being two dimensional the level sets of the period will be given by $\tau=\tau(x,y)$. If we change the coordinates in a neighbourhood of the origin such that $\tau=\varepsilon_1 \tilde{x}^2 + \varepsilon_2 \tilde{y}^2$ with $\varepsilon_j =\pm 1$,where the sign depends on the details of $h$ and $H$. One can give the following definition:
\begin{dfn}
The level sets of the period $\tau$  can be of two types:
\begin{enumerate}
\item \emph{elliptic} when $\varepsilon_1 \varepsilon_2=1$. In that case the level sets of the period form circles and $\tau$ increases or decreases monotonically with increasing radius.
\item \emph{hyperbolic} when $\varepsilon_1 \varepsilon_2=-1$. Here the level sets of the period form two families of hyperbolae, one family with positive increasing $\tau$ and one with negative decreasing $\tau$.
\end{enumerate}
\end{dfn}

Now we can prove the following proposition:

\begin{prop}
Depending on the quartic terms of the Hamiltonian function (and quadratic terms of the function $g$), among the three dimensional surface of symmetric periodic solutions near the equilibrium point, then the  level sets of $\tau$ are  elliptic when $n^2>c^2+d^2$ or hyperbolic when $n^2<c^2+d^2$.
\end{prop}

\begin{proof}
As discussed in the proof of the existence of symmetric periodic solutions, $\tau(x,y)$ can be calculated using the equation $g(z,\tau)=0$, with $z=x+iy$. Using our variables $N,C$ and $D$ and depending on the quadratic terms of that equation we have
\begin{align*}
g(N,C,D,\tau)&=0\\
nN+cC+dD+....&=-\frac{\tau}{2}\\
2n(x^2+y^2)+2c(x^2-y^2)-4d(xy)&=-\frac{\tau}{2}
\end{align*}

By the Morse Lemma the shape of $\tau(x,y)$ near the origin is given by the determinant
\begin{equation*}
 D=4^2(n^2-c^2-d^2).
 \end{equation*}
Therefore the family of periodic orbits is elliptic when $(n^2-c^2-d^2)>0$ or hyperbolic when $(n^2-c^2-d^2)<0$.
\end{proof}
Accordingly, one can easily deduce the following corollary.
\begin{cor}
The two dimensional families of non-symmetric periodic orbits given in Theorem $\ref{th1}$ exists if and only if the three dimensional family of symmetric periodic orbits is of elliptic type.
\end{cor}

\section{Anti-symplectic involution}
In this section  we analyse the problem of existence of periodic orbits in a Hamiltonian system which is equivariant under the action of an anti-symplectic involution $S$ (type AE in Table 1).  This was studied by J.~Li and Y.~Shi in \cite{b2}, but that paper contains a number of errors.  Firstly, the form of the Hamiltonian is not sufficiently general, for example the polynomial function $h= DN$ satisfies the symmetry of the problem but is not in the form assumed in \cite{b2}. This effects the results significantly and the general form of the Hamiltonian makes the calculations more difficult. There is also a serious error in the proof of their Lemma 5.3.  As a result we consider the problem anew.  We use a different basis from \cite{b2}, so the invariants and anti-invariants are different, and we determine a general formula for the reduced Hamiltonian. Firstly, we find that no symmetric periodic orbits can occur generically (opposite to the result claimed in \cite{b2}).  Secondly, we prove the existence of at least two  and at most 12 families of non-symmetric periodic solutions near the equilibrium point.

An immediate consequence of our assumptions is the Hamiltonian being $S$ anti-invariant (as pointed out in Table 1). By the normal forms given in Lemma \ref{l1}(ii) we have $\dim \ker \mathcal{L}=4$ i.e. $\ker \mathcal{L}\cong \mathbb{C}^2$. The bifurcation equation is given by the formula
\begin{align*}
\varphi:\mathbb{C}^2 \times \mathbb{R}\rightarrow \mathbb{C}^2\\
\varphi=2J \nabla_{z}h
\end{align*}
with the Hamiltonian
\begin{equation*}
h:\mathbb{C}^2 \times \mathbb{R}\rightarrow \mathbb{R},
\end{equation*}
where $J$ is the structure map. Now let's define the actions of $\mathbb{Z}_2^S \times S^1$ on $\mathbb{C}^2$ by \begin{align*}
S(z_1,z_2)&=(\bar{z_2},\bar{z_1})\\
\theta(z_1,z_2)&=(e^{i\theta}z_1, e^{-i\theta}z_2)
\end{align*}

Now we study the set of (anti-)invariants and find the appropriate   formula for $h$.
\begin{lem}
\label{lem2}
For $S^1\rtimes\mathbb{Z}_2^S$ acting on $\mathbb{C}^2$ as above, then
\begin{enumerate}
\item The $S^1\rtimes\mathbb{Z}_2$ invariant functions are generated by $N,C,D^2$ where
\[ N=|z_1|^2+ |z_2|^2 , C+iD= 2z_1 z_2 .\]
\item the $S^1$ invariant but $\mathbb{Z}_2$ anti-invariant functions are generated by $\delta,D$ where\[\delta=|z_1|^2- |z_2|^2.  \]
\end{enumerate}

\end{lem}
According to that the Hamiltonian $h$ will take the form
\begin{equation*}
h=\delta g^1(N,C,D^2,\tau)+ D g^2(N,C,D^2,\tau).
\end{equation*}
The bifurcation equation will be given by
\begin{align}
\label{no15}
\frac{\partial h}{\partial z_1} &= \bar{z_1} g^1+ \delta \frac{\partial g^1}{\partial z_1}-iz_2g^2+D \frac{\partial g^2}{\partial z_1}=0\\
\label{no16}
\frac{\partial h}{\partial z_2} &= -\bar{z_2} g^1+ \delta \frac{\partial g^1}{\partial z_2}-iz_1g^2+D\frac{\partial g^2}{\partial z_2}=0
\end{align}

\subsection{Symmetric Periodic Orbits}
Symmetric periodic solutions of that equivariant  Hamiltonian system lie in the set $\mathrm{Fix} S=\{(z,\bar{z}),z\in \mathbb{C}\}$. Moreover, by anti-invariance, that is $h\circ S=-h$, all symmetric solutions will be in the level set $h=0$. In order to get the symmetric periodic solutions we need to solve the bifurcation equation calculated in $\mathrm{Fix S}$. Consequently, one needs to solve \eqref{no15} and \eqref{no16}
with conditions: $\delta=D =0$ and $N=C$. Thus,
\begin{align}
\label{no17}
\bar{z_1}g^1-iz_2 g^2&=0\\
\label{no18}
-\bar{z_2}g^1-iz_1 g^2&=0.
\end{align}
By multiplying \eqref{no17} by $z_1$ and \eqref{no18} by $z_2$ we get
\begin{align}
|z_1|^2 g^1 -iz_1 z_2 g^2&=0\\
-|z_2|^2 g^1-iz_1 z_2 g^2&=0
\end{align}
Adding and subtracting these two equations yields
\begin{align*}
\delta g^1 -i(C+iD)g^2&=0\\
N g^1&=0.
\end{align*}
With the conditions $\delta=D=0$ we have
\begin{align*}
  C g^2&=0\\
  N g^1&=0.
\end{align*}
Since we are looking for nonzero solutions then $N=C\neq0$ and therefore, solutions are common zeros of $g^1$ and $g^2$ in a neighbourhood of the origin. But $g^1$ and $g^2$ are independent functions and generically the only common zero in a neighbourhood of the origin is $0$ itself. As a result there are no symmetric periodic orbits for the given Hamiltonian system.
\begin{rem}
Another way to see the non-existence of symmetric solutions in that system is by using a Liapunov function. Consider the Hamiltonian given by the formula $H=\delta(a_1+b_1N+c_1C+\cdots)+D(a_2+b_2N+c_2C+\cdots)$. Restricting the Hamiltonian system on the two dimensional invariant space $\mathrm{Fix} S$ gives
\begin{align*}
\dot{x}&=2y\left(a_1+2(b_1+c_1)(x^2+y^2)+\cdots\right)+2x\left(a_2+2(b_2+c_2)(x^2+y^2)+\cdots\right) \\
\dot{y}&=-2x\left(a_1+2(b_1+c_1)(x^2+y^2)+\cdots\right)+2y\left(a_2+2(b_2+c_2)(x^2+y^2)+\cdots\right)
\end{align*}
Easy computations show that the eigenvalues of the linear system are $\lambda=2(a_2\pm a_1 i)$. In order to get periodic orbits we should have $a_2=0$ and the system would be written as
\begin{align*}
\dot{x}&=2y\left(a_1+2(b_1+c_1)(x^2+y^2)+\cdots\right)+2x\left(2(b_2+c_2)(x^2+y^2)+\cdots\right) \\
\dot{y}&=-2x\left(a_1+2(b_1+c_1)(x^2+y^2)+\cdots\right)+2y\left(2(b_2+c_2)(x^2+y^2)+\cdots\right).
\end{align*}
Consider as Liapunov function $V=x^2+y^2$. Differentiating $V$ in the direction of the Hamiltonian vector field yields
\begin{align*}
\dot{V}&=2x\dot{x}+2y\dot{y}\\
&=8(x^2+y^2)^2(b_2+c_2).
\end{align*}
The number $b_2+c_2$ is generically non-zero and therefore $\dot{V}$ is non-zero. This means the sign of $\dot{V}$ (either positive or negative) is constant along any trajectory, so that the trajectory cannot be closed.  Thus, the system does not have any symmetric periodic orbits.
\end{rem}
\subsection{Non-symmetric Periodic Orbits}
For this case we only need to solve the pair \eqref{no15} and \eqref{no16} without any extra conditions. Multiplying \eqref{no15} by $z_1$ and \eqref{no16} by $z_2$ gives
\begin{multline}
\label{no19}
|z_1|^2g^1+\delta\left(g^1_N |z_1|^2+g^1_C z_1z_2
+ g^1_{D^2} 2D(-iz_1z_2)\right)-iz_1z_2g^2\\[12pt]
+D\left(g^2_N |z_1|^2+g^2_C z_1z_2 + g^2_{D^2} 2D(-iz_1z_2)\right)=0
\end{multline}
\begin{multline}
\label{no20}
-|z_2|^2g^1+\delta\left(g^1_N |z_2|^2+g^1_C z_1z_2
+ g^1_{D^2} 2D(-iz_1z_2)\right)-iz_1z_2g^2+\\[12pt]
+D\left(g^2_N |z_2|^2+g^2_C z_1z_2 + g^2_{D^2} 2D(-iz_1z_2)\right)=0
\end{multline}
By adding theses two equations we have
\begin{multline}
\label{no21}
\delta \left(g^1+Ng^1_N+(C+iD)g^1_C+2g^1_{D^2}(-iD)(C+iD)\right)-i(C+iD)g^2\\[12pt]
+D\left(Ng^2_N+(C+iD)g^2_C+2g^2_{D^2}(-iD)(C+iD)\right)=0
\end{multline}
The real and imaginary parts of equation \eqref{no21} are
\begin{equation}
\label{no22}
\delta \left(g^1+Ng^1_N+Cg^1_C+2g^1_{D^2}D^2\right)+Dg^2+D\left(Ng^2_N+Cg^2_C+2g^2_{D^2}D^2\right)=0
\end{equation}
\begin{equation}
\label{no23}
\delta \left(Dg^1_C-2g^1_{D^2}CD\right)-Cg^2+D\left(Dg^2_C-2g^2_{D^2}CD\right)=0
\end{equation}
The last equation to be considered comes from subtracting \eqref{no20} from \eqref{no19} and it will take the form
\begin{equation}
\label{no24}
Ng_1+\delta^2g^1_N+D\delta g^2_N=0
\end{equation}
This means finding non-symmetric solutions of the Hamiltonian system will be by solving the triple \eqref{no22},\eqref{no23} and \eqref{no24}. Clearly the system is singular at the origin and can be studied using a blow-up method. 
For that purpose define the new coordinates $(u,v,w,t,x)$ by
$$N=rv,\quad
C=ru,\quad
D=rw,$$
$$\tau=rt,\quad
\delta=rx,
$$
combined together by the relation $v^2=u^2+w^2+x^2$ according to the relation $N^2=\delta^2+C^2+D^2$. Substituting these new coordinates in \eqref{no22},\eqref{no23} and\eqref{no24} gives
\[r\left(vg^1+rx^2g^1_N+rxwg^2_N\right)=0\]
\[r\left(x(g^1+rvg^1_N+rug^1_C+2r^2w^2g^1_{D^2})+w(g^2+rvg^2_N+rug^2_C+2r^2w^2g^2_{D^2})\right)=0\]
\begin{equation}
\label{no25}
r\left(x(g^1_Crw-2r^2wug^1_{D^2})-ug^2+w(rwg^2_C-2r^2uwg^2_{D^2})\right)=0
\end{equation}
We are interested in the non-zero solutions, i.e.\ $r\neq0$. The first step is to divide by the common power of $r$ in these equations and the second step is to apply the implicit function theorem. For simplicity we can write the Taylor series for the functions $g^1$ and $g^2$ as
\begin{align*}
g^1&=\displaystyle\frac{\tau}{2}+a_1N+c_1C+d_1D^2+\cdots\\
g^2&=b_2\tau+a_2N+c_2C+d_2D^2+\cdots
\end{align*}
which with the new coordinates take the form 
\begin{align*}
g^1=r\displaystyle\bar{g}^1&=r(\displaystyle\frac{t}{2}+a_1v+c_1u+d_1rw^2+\cdots)\\
g^2=r\displaystyle\bar{g}^2&=r(b_2t+a_2v+c_2u+d_2rw^2+\cdots)
\end{align*}
Accordingly, the system \eqref{no25} will be written as  
\[r^2\left(v\bar{g}^1+x^2\bar{g}^1_v+xw\bar{g}^2_v\right)=0\]
\[r^2\left(x(\bar{g}^1+v\bar{g}^1_v+u\bar{g}^1_u+2rw^2\bar{g}^1_{rw^2})+w(\displaystyle\bar{g}^2+v\bar{g}^2_v+u\bar{g}^2_u+2rw^2\bar{g}^2_{rw^2})\right)=0\]
\begin{equation}
\label{no26}
r^2\left(x(w\bar{g}^1_u-2rwu\bar{g}^1_{rw^2})-u\displaystyle\bar{g}^2+w(w\bar{g}^2_u-2ruw\bar{g}^2_{rw^2})\right)=0
\end{equation}
Note here that $g^1_N=\bar{g}^1_v$ etc.  Dividing by $r^2$ and substituting $r=0$ yields
\begin{equation*}
v(\frac{t}{2}+a_1v+c_1u)+a_1x^2+a_2xw=0
\end{equation*}
\begin{equation*}
x(\frac{t}{2}+2a_1v+2c_1u)+w(b_2t+2a_2v+2c_2u)=0
\end{equation*}
\begin{equation}
\label{no27}
c_1xw-u(b_2t+a_2v+c_2u)+c_2w^2=0.
\end{equation}
Clearly, the system can not be solved by the implicit function theorem at this point in the argument. As a result we will use a different technique as illustrated in the next section. We will show that \eqref{no27} has non-degenerate solutions, then apply a continuation argument to show \eqref{no26} has solutions when $r>0$. Adding the relation between the variables $N,C,D$ and $\delta$ gives us the system

\begin{equation*}
v(\frac{t}{2}+a_1v+c_1u)+a_1x^2+a_2xw=0
\end{equation*}
\begin{equation*}
x(\frac{t}{2}+2a_1v+2c_1u)+w(b_2t+2a_2v+2c_2u)=0
\end{equation*}
\begin{equation*}
c_1xw-u(b_2t+a_2v+c_2u)+c_2w^2=0
\end{equation*}
\begin{equation}
\label{no40}
u^2+w^2+x^2-v^2=0.
\end{equation}

 First of all we want to count the number of all solutions of the system \eqref{no40}. For that purpose we need the following theorem.

\begin{thm}[Bezout's theorem]
\label{no41}
Suppose $n$ homogeneous polynomials on $\mathbb{C}$ in $n+1$ variables, of degrees  $d_1,d_2,..,d_n$, that define $n$ hypersurfaces in the projective space of dimension $n$. If the number of intersection points of the hypersurfaces is finite, then this number is $d_1d_2..d_n$ if the points are counted with their multiplicity.
\end{thm}
For more details and proof see for example \cite{b7}.

The system \eqref{no40} consists of four homogeneous equations  each of degree two with five variables. According to Bezout's Theorem we have $16$ complex solutions for that system and can divide them into two main types: solutions when $v=0$ and solutions when $v\neq0$.
\subsubsection{\textbf{Solutions when $\boldsymbol{v=0}$}}
In that case, algebric calculations  give  a total of three different  solutions:
\begin{enumerate}
 \item $\{t \in \mathbb{R}, u = 0, w = 0, x = 0\}$,
 \item $ \{t = \mp 2i\frac{c_2}{b_2}w,\, u =  \pm iw, \,w \in\mathbb{R}, \,x = 0\}.$
 \end{enumerate}
 Now we want to study the multiplicity of each solution. Consider the Jacobian matrix for the system \eqref{no40} with respect to $v,t,u,w,x$
 \[J=
 \left(\begin{smallmatrix}
\frac{1}{2}t+2a_1v+c_1u & \frac{1}{2}v & c_1v & a_2x & 2a_1x+a_2w\\
2a_1x+2a_2w & \frac{1}{2}x+b_2w & 2c_1x+2c_2w & 2a_2v+b_2t+2c_2u & \frac{1}{2}t+2a_1v+2c_1u \\
-a_2u & -b_2u & -a_2v-b_2t-2c_2u & c_1x+2c_2w & c_1w\\
-2v & 0& 2u & 2w & 2x\\
\end{smallmatrix}\right)
\]

Substituting the values of the first solution and the condition $v=0$ in the Jacobian matrix yields
\[J\mid_{v=0,sol.1}=
\left( \begin{matrix}
\frac{1}{2}t & 0 & 0 &0 & 0\\
0 & 0 & 0 & b_2t& \frac{1}{2}t\\
0& 0 & -b_2t& 0 & 0\\
0 & 0& 0 & 0 & 0\\
\end{matrix}
\right)
\]
  To get the appropriate square submatrix we eliminate the second column because $t$ is non-zero and get
 \[ J_1=
 \left( \begin{matrix}
\frac{1}{2}t & 0 &0 & 0\\
0 & 0 & b_2t& \frac{1}{2}t\\
0&  -b_2t& 0 & 0\\
0 &0 & 0 & 0\\
\end{matrix}
\right)
 \]
   This matrix is of rank three and therefore this first solution is not simple. To study its multiplicity we need to study the behaviour of system \eqref{no40} near a solution point for example say $(v,t,u,w,x)=(0,2,0,0,0)$. Consider the system
   \begin{equation*}
v(1+a_1v+c_1u)+a_1x^2+a_2xw=\varepsilon_1
\end{equation*}
\begin{equation*}
x(1+2a_1v+2c_1u)+w(2b_2+2a_2v+2c_2u)=\varepsilon_2
\end{equation*}
\begin{equation*}
c_1xw-u(2b_2+a_2v+c_2u)+c_2w^2=\varepsilon_3
\end{equation*}
\begin{equation}
\label{no42}
u^2+w^2+x^2-v^2=\varepsilon_4.
\end{equation}
Near the point $(v,t,u,w,x)=(0,2,0,0,0)$ the first equation can be solved by the implicit function theorem for $v$, the second for $x$ and the third equation for $u$. As a result we end up with solving the equation
\begin{equation*}
w^2+f(w)=\varepsilon_4,
\end{equation*}
where $f(w)$ is a function constructed by substituting the solutions from  the implicit function theorem in equation \eqref{no42}. Clearly $f$ is of order greater than one. So, the least order  coefficient is $w^2$ and so the studied solution is of multiplicity two.\\

   Regarding the multiplicity of the second and third solution we should assume that $w\neq0$ for a non-zero solution; for simplicity let $w=1$. The Jacobian matrix will take the form
   \[J\mid_{v=0,w=1,sol.2}=
\left( \begin{matrix}
\mp c_2i/b_2\pm c_1 i & 0 & 0 &0 & a_2\\
2a_2& b_2 & 2c_2 & 0& \mp c_2 i/b_2\pm 2 c_1 i\\
\mp a_2 i& \mp b_2 i & 0&2c_2 & c_1\\
0 & 0& \pm 2i & 2 & 0\\
\end{matrix}
\right)
\]
Since $w=1$, we can omit the $w$- column and get
 \[J_2=
\left( \begin{matrix}
\mp c_2i/b_2\pm c_1i & 0 & 0& a_2\\
2a_2& b_2 & 2c_2 &\mp c_2i/b_2\pm 2 c_1i\\
\mp a_2i& \mp b_2i & 0&c_1\\
0 & 0& \pm 2i & 0\\
\end{matrix}
\right)
\]
\[ \det{J_2}= -2(a_2^2 b_2^2+b_2^2 c_1^2-2 b_2 c_1 c_2+c_2^2)/b2.\]
We can assume that this result is non-zero and therefore the second and the third solutions are simple. We conclude that the case $v=0$ corresponds to four solutions where the first solution is doubled but the others are of multiplicity one. Note that $v=0$ implies $N= |z_1|^2+ |z_2|^2=0$. Thus, these four solutions won't be counted as periodic solutions of the given system, but will help us find out how many non-zero periodic solutions there are.

\subsubsection{\textbf{Solutions when $\boldsymbol{v\neq0}$}}
 There remain $12$ solutions for the case $v\neq0$ according to Bezout's theorem . The following proposition guarantees a minimum of two real solutions for the system \eqref{no40}.
\begin{prop}
\label{prop1}
For any choice of coefficients $\{a_1,a_2,b_2,c_1,c_2\}$ the two points

\[\{v\in {\mathbb{R}}^*, t=-4 a_1 v, u=w=0, x=\pm v\},\] satisfy equation \eqref{no40} , when $v\neq0$.
\end{prop}
\begin{proof}
Straightforward calculations yield the result.
\end{proof}

 In order to find out more about the maximum number of real solutions we can find we will use a numerical approach. We choose various values for the constants in the system \eqref{no40} and then solve the equations using Maple. Since we are interested in solutions with $v\neq0$, we put $v=1$ for simplicity.  These numerical calculations suggest that the system can have a maximum of eight real solutions, including the two analytic solutions given by Proposition \ref{prop1}. In addition, there are examples of systems with four or six real solutions. Our aim is to prove  that for each of these cases, the solutions are non-degenerate. Then, under any perturbation of the set of coefficients there still exist (nearby) real solutions (i.e.\ periodic solutions). In the following we study an example of each set of coefficients that has two, four, six or eight  real solutions for the studied system \eqref{no40}. Then, we check their non-degeneracy conditions. Note that all numbers are rounded to four decimal digits.

\begin{exm}[A system with two real solutions]
Consider the set \[R=\{a_1 = 1, a_2 = 5, b_2 = 1, c_1 = 2, c_2 = 2, v = 1\}.\]
The corresponding system has only two real solutions
\[\{t = -4, u = 0, w = 0, x =\pm v= \pm 1\},\]
which are those given in Proposition \ref{prop1}. The remaining $10$ solutions are non-real. In order to check the non-degeneracy condition, we need to study the proper submatrix of $J$ for each solution and ensure that its determinant is non-zero.
Substituting the values given in $R$ and the two solutions in $J$ yields
\[J_1=
\left( \begin{matrix}
0&0.5&2&\pm 5&\pm 2\\
\pm 2&\pm 0.5&\pm 4&6&0\\
0&0&-1&\pm 2&0\\
-2&0&0&0&\pm 2\\
\end{matrix}
\right).
\]
Since $t\neq0$, we omit the $t-$column and we have the submatrix
\[J_{11}=
\left( \begin{matrix}
0&2&\pm 5&\pm 2\\
\pm 2&\pm 4&6&0\\
0&-1&\pm 2&0\\
-2&0&0&\pm 2\\
\end{matrix}
\right),
\]
\[\det{J_{11}}=\pm20.\]
Therefore, these two solutions are non-degenerate.
 \end{exm}
A similar argument is used in the remaining examples to prove the non-degeneracy of solutions in each case.

\begin{exm}[A system with four real solutions]
Let the set of coefficients in the system \eqref{no40} be
\[R=\{a_1 = 1, a_2 = 5, b_2 = -2, c_1 = 2, c_2 = 2, v = 1\}.\]
the associated system has $8$ non-real solutions and only four real solutions and the real ones are :
\begin{enumerate}
\item $\{t = -4, u = 0, w = 0, x =\pm v= \pm 1\}$
\item $\{t = 1.5602, u = -0.9681, w = \pm 0.0855, x =\pm  0.2354\}$
\end{enumerate}

Substituting $R$ and the first two solutions in the matrix $J$ we  get
 \[J_1=
\left( \begin{matrix}
0 &0.5&2&\pm5&\pm2\\
\pm2&\pm0.5&\pm4&18&0\\
0&0&-13&\pm2&0\\
-2&0&0&0&\pm2\\
\end{matrix}
\right)
\]
Now we can choose the submatrix $J_{11}$ by omitting the second column because $t$ is non-zero and we find its determinant to be $\det J_{11} = \pm692\neq0$.
In the same way we  can study the third and fourth solutions to get
\[J_{22}=
\left(
 \begin{matrix}
0.8439&2&\pm 1.1772&\pm0.8985\\
\pm1.3262&\pm1.2839&3.0071&-1.0924\\
4.8406&1.9929&\pm0.8130&\pm0.1711\\
-2&-1.9362&\pm0.1711&\pm0.4709\\
\end{matrix}
\right)\\
\]
We have $\det{J_{22}}=\pm 35.6351\neq0$.

Since the determinants are non-zero, all four solutions are non-degenerate and we can find an open set of coefficients that give four real solutions.
\end{exm}
\begin{exm}[A system with six real solutions]
Let
\[R=\{a_1 = -2, a_2 = -11, b_2 = -5, c_1 = 1, c_2 = 2, v = 1\}.\]
The system \eqref{no40} with those coefficients has only six real solutions:
\begin{enumerate}
\item $\{t = 8, u = 0, w = 0, x = \pm v=\pm 1\}$
\item $\{t = -2.5592, u = 0.0346, w = \pm 0.4980,x = \mp 0.8665\}$
\item $\{t = -3.7663, u = 0.1529, w =\pm 0.8984,x = \mp 0.4118\}$
\end{enumerate}
The non-degeneracy of the above solutions can be studied in pairs. Firstly, we study the determinant of the appropriate matrix $J_{11}$ associated to the  first and second solutions.

\[J_{11}=
\left( \begin{matrix}
0&1&\mp11&\mp4\\
\mp4&\pm2&-62&0\\
0&51&\pm1&0\\
-2&0&0&\pm2\\
\end{matrix}
\right)\\
\]
\[ \det{J_{11}}=\mp20816.\]

Similarly, for the rest of solutions we have
\[J_{22}=
\left( \begin{matrix}
-5.2450&1&\pm9.5314&\mp2.0120\\
\mp7.4900&\pm0.2590&-9.0658&-5.2104\\
0.3804&-1.9342&\pm1.1255&\pm0.4980\\
-2&0.0692&\pm0.9960&\mp1.7330\\
\end{matrix}
\right)\\
\]
\[ \det{J_{22}}=\mp164.8123\]

\[J_{33}=
\left( \begin{matrix}
-5.7303&1&\pm4.5299&\mp8.2346\\
\mp18.1165&\pm2.7698&-2.5567&-5.5774\\
1.6818&-8.4433&\pm3.1816&\pm0.8984\\
-2&0.3058&\pm1.7967&\mp0.8236\\
\end{matrix}
\right)\\
\]
\[ \det{J_{33}}=\pm1827.2294.\]
As a result all real solutions of this case are non-degenerate
\end{exm}

We end with an example of a system with eight real solutions, which is the largest number of real solutions we found using numerical calculations.
\begin{exm}[A system with eight real solutions]
Let
\[R=\{a_1 = 1, a_2 = -4, b_2 = -1, c_1 = 1, c_2 = 2, v = 1\}\]
The corresponding real solutions are only eight and they are
\begin{enumerate}
\item$\{t = -4, u = 0, w = 0, x = \pm v=\pm 1\}$
\item$\{t = -4.9432, u = -0.2615, w =\pm 0.2274, x = \mp 0.9380\}$
\item$\{t = -2.8537, u = 0.8527, w = \pm 0.4155, x =\pm 0 .3165\}$
\item$\{t = -6.4260, u = 0.2940, w = \pm 0.8063, x = \mp 0.5133\}$.
\end{enumerate}
The non-degeneracy conditions are
\[J_{11}=
\left( \begin{matrix}
0&1&\mp4&\pm2\\
\pm2&\pm2&-4&0\\
0&0&\pm1&0\\
-2&0&0&\pm2\\
\end{matrix}
\right)\\
\]
\[ \det{J_{11}}=\pm4\]
\[J_{22}=
\left( \begin{matrix}
-0.7331&1&\pm3.7521&\mp2.7857\\
\mp3.6953&\mp0.9664&-4.1029&-0.9947\\
-1.0461&0.1029&\mp0.0284&\pm0.2274\\
-2&-0.5231&\pm0.4548&\mp1.8760\\
\end{matrix}
\right)\\
\]
\[ \det{J_{22}}=\mp13.8083\]
\[J_{33}=
\left( \begin{matrix}
1.4259&1&\mp1.2659&\mp1.0293\\
\mp2.6915&\pm2.2951&-1.7353&2.2786\\
3.4110&-2.2647&\pm1.9787&\pm0.4155\\
-2&1.7055&\pm0.8311&\pm0.6329\\
\end{matrix}
\right)\\
\]
\[ \det{J_{33}}=\mp43.7450\]

\[J_{44}=
\left( \begin{matrix}
-0.9190&1&\pm2.0533&\mp4.2517\\
\mp7.4767&\pm2.1984&-0.3979&-0.6249\\
1.1761&-3.6021&\pm2.7117&\pm0.8063\\
-2&0.5881&\pm1.6125&\mp1.0266\\
\end{matrix}
\right)\\
\]
\[ \det{J_{44}}=\pm111.6657\]
Therefore, all eight solutions are non-degenerate.
\end{exm}
\subsection{Conclusion}
Bezout's theorem guaranteed a total of $12$ solutions for the case $v\neq0$, but numerical calculations found at most eight of them to be real (and at least two). The last thing to consider is the effect of the addition of higher order terms to the system \eqref{no40} when solving by the implicit function theorem. We will choose one of the previous examples and prove the existence of periodic orbits in that system and the rest can be done in the same way.

We consider the solution point $(t,u,w,v,x,r)=(-4,0,0, 1,1,0)$ as a candidate. We want to apply the implicit function theorem on the system in a neighbourhood of that point. Note that the functions $g^1,g^2$ are given by
\begin{align}
\label{no43}
g^1(N,C,D^2,\tau)&=\displaystyle\frac{\tau}{2}+a_1N+c_1C+d_1D^2+e_1N^2+f_1NC+g_1N\tau+\cdots,\\
\label{no44}
g^2(N,C,D^2,\tau)&=b_2\tau+a_2N+c_2C+d_2D^2+e_2N^2+f_2NC+g_2N\tau+\cdots
\end{align}
In our new coordinates \eqref{no43} and \eqref{no44} will take the form
\begin{align}
g^1(N,C,D^2,\tau)&=r[\displaystyle\frac{t}{2}+a_1v+c_1u+d_1rw^2+e_1rv^2+f_1rvu+g_1rvt+\cdots],\\
g^2(N,C,D^2,\tau)&=r[b_2t+a_2v+c_2u+d_2rw^2+e_2rv^2+f_2rvu+g_2rvt+\cdots],
\end{align}
therefore, the matrix formula associated to the implicit function theorem  calculated at the point $(-4,0,0, 1,1,0)$ will be
\[
\left(
\begin{array}{cccc|c}
 0&2&5&2&3e_1-8g_1\\
 2&4&6&0&3e_1-8g_1\\
 0&-1&2&0&0\\
 -2&0&0&2&0\\
  \end{array}
   \right) =
   \left(
 \begin{array}{c|c}
 X&Y
  \end{array}
 \right).
 \]
The matrix $X$ is invertible and by the implicit function theorem we can solve $v,u,w,x$ as functions of $r$. The linear part of the Taylor series of those solutions is determined by the matrix
\[
X^{-1}Y=\left(
\begin{array}{cccc}
 7/5&-9/10&-4/5&-7/5\\
 -2/5&2/5&-1/5&2/5\\
-1/5&1/5&2/5&1/5\\
7/5&-9/10&-4/5&-9/10\\
  \end{array}
   \right)
   \left(
   \begin{array}{c}
   3e_1-8g_1\\
   3e_1-8g_1\\
   0\\
   0\\
   \end{array}
   \right).
   \]
Those solutions can be written as functions of $r$ as follows
\begin{align*}
v(r)&=1-\frac{1}{2}(3e_1-8g_1)r+h.o.t.\\
x(r)&=1-\frac{1}{2}(3e_1-8g_1)r+h.o.t.\\
v&=w=0.
\end{align*}
Converting back to our basic coordinates $N,C,D,\delta$ gives
\begin{align*}
N&=rv=r-\frac{1}{2}(3e_1-8g_1)r^2+h.o.t.\\
\delta&=rx=r-\frac{1}{2}(3e_1-8g_1)r^2+h.o.t.\\
C&=D=0
\end{align*}
This curve of solutions gives a one parameter family of periodic orbits for the equivariant Hamiltonian system. Similarly one can prove the existence of one parameter family of periodic solutions for each case studied before because of their non-degeneracy conditions.
 Accordingly, we  state the following  result.

\begin{thm}
Consider an equilibrium point $0$ of a $C^\infty $ equivariant Hamiltonian vector field $f$, with the
the symmetry $S$ acting anti-symplectically and $S^2 =I$. Assume that the linear Hamiltonian vector field
$L$ has two pairs of purely imaginary eigenvalues $\pm i $ and no other eigenvalues of the form $\pm ki,
k \in \mathbb{Z}$. The reduced Hamiltonian is in the form $h=\delta g^1(N,C,D^2,\tau)+ D g^2(N,C,D^2,\tau).$ Then
\begin{enumerate}
\item For an open dense set of coefficients $(a_1,a_2,b_2,c_1,c_2)$ there exists a neighbourhood of $0$ with no symmetric periodic orbits and at least two and at most $12$ non-symmetric periodic solutions of the equivariant Hamiltonian system .

\item There exist open sets of coefficients $U_i$ $(i=1,2,3,4)$, such that for coefficients in $U_i$ there are precisely $2i$ non-symmetric periodic orbits of period close to $2\pi$ as they tend to zero.
\end{enumerate}
\end{thm}

\section{ The combined case $\mathbb{Z}_2^R \times\mathbb{Z}_2^S$}
It is natural at this point to ask about periodic orbits of a system possessing both the symmetries studied above. Consider now a reversible equivariant Hamiltonian system  under the action of the group $G=\mathbb{Z}_2^R \times\mathbb{Z}_2^S$, where $R$ and $S$ are the involutions defined in Section 5 and Section 6 respectively.  In this section we prove the existence of families of periodic solutions in a neighbourhood of the origin in that system.

On $\mathbb{C}^2$ the reduced Hamiltonian is a special case of the Hamiltonian in Section 5 and it takes the form
\[h(z_1,z_2,\tau)=\delta g(N,C,D^2,\tau).\]
Accordingly, the bifurcation equation will be
\begin{equation*}
\delta[g+N g_N+Cg_C+2D^2g_{D^2}]=0,
\end{equation*}

\begin{equation*}
Ng+[N^2-C^2-D^2]g_N=0,
\end{equation*}

\begin{equation}
\label{no46}
\delta D[g_C-2Cg_{D^2}]=0.
\end{equation}

Now we classify the solutions according to their symmetry type.

\subsection{Periodic orbits in the conical subspace \texorpdfstring{$\boldsymbol{\delta=0}$}{delta=0}}
Substituting $\delta =0$ in the system of equations \eqref{no46} yields
\begin{equation}
\label{no47}
Ng=0.
\end{equation}
For $R$ symmetric solutions one needs to solve \eqref{no47} in $\mathrm{Fix}R$. This implies
\[g(z,\tau)=0,\]
which can be solved for $\tau=\tau(z), z\in \mathrm{Fix}R$ by the implicit function theorem. This means any periodic orbit in the subspace $\delta=0$ has symmetry $R$. Solving equation\eqref{no47} for $z\in \mathrm{Fix}S$ gives one periodic orbit of symmetry $S$ and it is therefore $SR$ symmetric. Moreover, solving equation\eqref{no47} for $z\in \mathrm{Fix}(S,\pi)$ gives another orbit with symmetry $S$.

\subsection{Periodic orbits in \texorpdfstring{$\boldsymbol{\delta\neq0}$}{delta<>0}}
It remains to study the existence of $SR$ solutions which lie in the subset $\delta\neq0$. Clearly $\mathrm{Fix}RS=\{(z_1,z_2)\mid z_1,z_2\in \mathbb{R}\}$ which implies $D=0$ and therefore the system \eqref{no46} takes the form
\begin{equation*}
g+N g_N+Cg_C=0,
\end{equation*}
\begin{equation}
\label{no49}
N g+[N^2-C^2]g_N=0.
\end{equation}
Eliminating $g$ from both equations gives
\begin{equation}
\label{no48}
C(Ng_C+Cg_N)=0.
\end{equation}
If $C=0$, then by the fact $\frac{\partial g}{\partial\tau}(0)=\frac{1}{2}$ we can solve using the implicit function theorem. Now if $C\neq 0$, and $g_N(0)=n,g_C(0)=c$ are not both zero then, the system\eqref{no49} can be solved by the implicit function theorem. By the argument used in Theorem 5.2 we conclude that $SR$ periodic solutions exist when $n^2-c^2>0$. The following theorem describes the families of periodic solutions exist in this system.

\begin{thm}
Consider a symmetric equilibrium $0$ of a $\mathbb{Z}_2^R\times \mathbb{Z}_2^S$ reversible equivariant Hamiltonian vector field $f$ where $R$  is a reversing involution acting symplectically and $S$ is an involution acting anti-symplectically. Suppose that $Df(0)$ has two purely imaginary pairs of eigenvalues $\pm i$ with no other eigenvalues of the form $\pm ki,k \in \mathbb{Z}$. Also, denote $g_N(0)=n$ and $g_C(0)=c$. Then,
\begin{enumerate}
 \item there exists a two-parameter family of $R$ symmetric periodic solutions  in the conical subspace $\delta =0$ with two of them having extra symmetry $S$. The  period of all orbits tends to $2\pi$ as they approach the equilibrium.
 \item  there exist two Liapunov centre families of $SR$ symmetric periodic solutions in the open subset $\delta\neq0$ provided that $n^2-c^2>0$---one with $\delta>0$ and one with $\delta<0$. These two families are exchanged by both involutions $R$ and $S$.
\end{enumerate}
\end{thm}

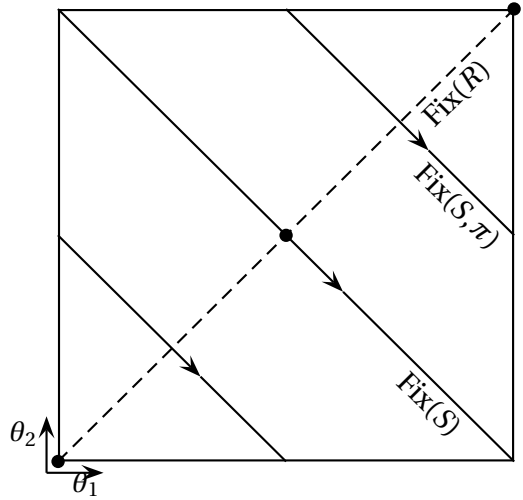
\begin{figure}
\centering
\psset{unit=0.75}
\begin{pspicture}(-4.5,-4.5)(4,4)
\psdots[linewidth=1.5pt](0,0)(4,4)(-4,-4)
 \psset{arrowsize=6pt}
\psframe(-4,-4)(4,4)
\psline[linestyle=dashed](-4,-4)(4,4) \rput{45}(3,2.5){\large$\mathrm{Fix}(R)$}
\psline{->}(-4,4)(1,-1) \psline(1,-1)(4,-4)
\psline{->}(-4,0)(-1.5,-2.5) \psline(-1.5,-2.5)(0,-4)
\psline{->}(0,4)(2.5,1.5) \psline(2.5,1.5)(4,0) \rput{-45}(3,0.5){\large$\mathrm{Fix}(S,\pi)$}
\rput{-45}(2.5,-3){\large$\mathrm{Fix}(S)$}
 \psline{->}(-4.2,-4.2)(-4.2,-3.2) \rput(-4.6,-3.5){$\theta_2$}
 \psline{->}(-4.2,-4.2)(-3.2,-4.2) \rput(-3.5,-4.4){$\theta_1$}
\end{pspicture}
\caption{Fixed point spaces on the torus $(\theta_1,\theta_2)$.}
\label{fig:no34}
\end{figure}

Finally we illustrate the relation between  fixed point spaces of the involutions $R,S,SR$ and $(S,\pi)$ geometrically.  Buzzi and Lamb \cite{b1} show that  the intersection between the cone $\delta=0$ and the unit sphere in $\mathbb{C}^2$ is a torus $T$ parametrized by two angles $(\theta_1,\theta_2)$ and draw $\mathrm{Fix}R$ on $T$. In addition to that we  show the intersection between $\mathrm{Fix}S$ and the torus $T$ is given by the line $\theta_2=-\theta_1$. Also, we plot $\mathrm{Fix}(S,\pi)=\{(\theta_1,\theta_2)=(\theta_1,\pi-\theta_1)\}$ on the torus. The last thing is to intersect $\mathrm{Fix}SR$ with $T$ which gives a total of two points $(0,0)$ and $(\pi,\pi)$ (shown as large dots in the figure).

\bibliographystyle{amsplain}

\begin{thebibliography}{33}

\bibitem{b6} R.~Abraham and J.E.~Mardsen,
\emph{Foundations of mechanics}, 2nd ed., Addison-Wesley, Reading MA, 1978.

\bibitem{Reem} R.\ Alomair, Periodic orbits in symmetric Hamiltonian systems.  Thesis, University of Manchester (in preparation).

\bibitem{b5} S.~Bochner,
\emph{Compact groups of differentiable transformations},
Ann.\ of Math.\ (2), \textbf{46} (1945), 372--381.

\bibitem{BH}
M.~Bosschaert, and H.~Han{\ss}mann,
\emph{Bifurcations in Hamiltonian systems with a reflecting symmetry.}
Qual.\ Theory Dyn.\ Syst.\ \textbf{12} (2013),  67--87.

\bibitem{b1}C.A.~Buzzi and J.S.W.~Lamb,
\emph{ Reversible Hamiltonian Liapunov centre Theorem},
Discrete Contin.\ Dyn.\ Syst.\ Ser.\ B, \textbf{5} (2005), 51--66.

\bibitem{b13}R.L.~Devaney,
\emph{Reversible diffeomorphisims and flows}, Trans.\ Amer.\ Math.\ Soc., \textbf{218} (1976), 89--113.

\bibitem{b4}M.~Golubitsky, J.E.~Mardsen, I.~Stewart and M.~Dellnitz,
\emph{The constrained Liapunov-Schmidt procedure and periodic orbits},
Normal forms and homoclinic chaos (Waterloo, ON, 1992). Fields Inst.\ Commun.\ \textbf{4}, AMS, Providence, RI, 1995, 81--127.

\bibitem{b10} M.~Golubitsky and D.G.~Schaeffer,
\emph{Singularities and groups in bifurcation theory .Vol.I}, Appl.\ Math.\ Sci., \textbf{51}, Springer-Verlag, New York, 1985.

\bibitem{b11} M.~Golubitsky, I.~Stewart and  D.G.~Schaeffer,
\emph{Singularities and groups in bifurcation theory. Vol.\ II}, Appl.\ Math.\ Sci., \textbf{69}, Springer-Verlag, New York, 1988.

\bibitem{b7} B.~Hassett,
\emph{Introduction to algebraic geometry}, Cambridge University Press,  Cambridge, 2007.

\bibitem{b3}I.~Hoveijn, J.~Lamb and R.~Roberts,
\emph{Normal forms and unfolding of linear systems in the eigenspaces of (anti-)automorphisims of order two},
 J.~Differential Equations, \textbf{190} (2003), 182--213.

\bibitem{b12}J.S.W.~Lamb and J.A.G.~Roberts,
\emph{Time-reversal symmetry in dynamical systems: a survey}, Physica D, \textbf{112} (1998),1--39.

\bibitem{b2} J.~Li and Y.~Shi,
\emph{The Liapunov centre Theorem for a Class of Equivariant Hamiltonian Systems},
 Abstr.\ Appl.\ Anal.,  \textbf{2012}, Article ID 530209, 12 pages.

\bibitem{b8} J.A.~Montaldi, R.M.~Roberts and I.N.~Stewart,
\emph{Periodic solutions near equilibria of symmetric Hamiltonian systems}, Phil.\ Trans.\ R.\ Soc.\ Lond.\ A, \textbf{325} (1988), 237--293.

\bibitem{b9} J.A.~Montaldi, R.M.~Roberts and I.N.~Stewart,
\emph{Existence of nonlinear normal modes of symmetric Hamiltonian systems}, Nonlinearity, \textbf{3} (1990), 695--730.



\end{thebibliography}

\end{document}